\documentclass[12pt]{article}

\newtheorem{theorem}{Theorem}

\def\qed{\begin{flushright} $\Box$ \end{flushright}}

\usepackage{amscd}
\usepackage{epsfig}
\usepackage{url}
\usepackage{amssymb}
\usepackage{amsmath}
\usepackage{amsfonts}

\def\Dbar{\leavevmode\lower.6ex\hbox to 0pt{\hskip-.23ex \accent"16\hss}D}

\def\bZ{{\mbox{\bf Z}}}

\def\paf{{\mbox{\rm PAF}}}

\begin{document}

{\bf\LARGE
\begin{center}
Construction of symmetric Hadamard matrices
\end{center}
}

\noindent
{\bf N. A. Balonin$^a$}, Dr. Sc., Tech., Professor, korbendfs@mail.ru \\
{\bf Y. N. Balonin$^a$}, tomaball@mail.ru \\
{\bf D. {\v{Z}}. {\Dbar}okovi{\'c}$^b$}, PhD, Distinguished Professor Emeritus, djokovic@uwaterloo.ca \\
{\bf D. A. Karbovskiy$^a$}, alra@inbox.ru \\
{\bf M. B. Sergeev$^a$}, Dr. Sc., Tech., Professor, 
mbsn@mail.ru \\
\noindent
${}^{a}$Saint-Petersburg State University of Aerospace Instrumentation, 
67, B. Morskaia St., 190000, Saint-Petersburg, Russian Federation \\
${}^{b}$University of Waterloo, Department of Pure Mathematics and Institute for Quantum Computing, Waterloo, Ontario, N2L 3G1, Canada \\

\begin{abstract}
{\bf Purpose:} To investigate more fully, than what was done in the past, the construction of symmetric Hadamard matrices of 
``propus type'', a symmetric variation of the Goethals-Seidel array characterized by necessary symmetry of one of the blocks and equality of two other blocks out of the total of four blocks.
{\bf Methods:} Analytic theory of equations for parameters of difference families used in the propus construction of symmetric Hadamard matrices, based on the theorems of Liouville and Dixon. Numerical method, due to the authors, for the search of two or three cyclic blocks to construct Hadamard matrices of two-circulant or propus type. This method speeds up the classical search of required sequences by distributing them into different bins using a hash-function. {\bf Results:} A wide collection of new symmetric Hadamard matrices was obtained and tabulated, according to the feasible sets of parameters. In addition to the novelty of this collection, we have obtained new symmetric Hadamard matrices of orders 92, 116 and 156. For the order 156, no symmetric Hadamard matrices were known previously. {\bf Practical relevance:} Hadamard matrices are used extensively in the problems of error-free coding, compression and masking of 
video information. Programs for search of symmetric Hadamard matrices and a library of constructed matrices are used in the 
mathematical network ``Internet'' together with executable 
on-line algorithms.

\end{abstract}

\section{Introduction}

In this paper we investigate some special features of symmetric 
Hadamard matrices. Let us recall that a Hadamard matrix is a $\{1,-1\}$-matrix {\bf H} of order $n$ whose columns (or rows) 
are mutually orthogonal
\begin{equation} \label{HM}
{\bf H}^T{\bf H}={\bf H}{\bf H}^T=n{\bf I},
\end{equation}
where {\bf I} is the identity matrix. This definition is due 
to Hadamard \cite{Hadam}, who pointed out the extremal property 
of the solutions of this quadratic equation (these matrices have 
the maximal possible absolute value of determinant among all 
complex matrices whose entries have modulus at most 1), and also 
the possibility that such matrices exist for all orders 
$n=4v$, $v$ integer.

As a rule, the search for Hadamard matrices is simplified by using special arrays, built from circulant blocks, i.e., matrices 
generated by cyclic shifts of the top row. As an example, we 
can mention the Williamson array \cite{JW} which makes use of four circulant matrices {\bf A}, {\bf B}, {\bf C}, {\bf D} and their negatives as blocks inside the globally non-symmetric array. The requirement that the blocks be symmetric works in some cases but not always. The first failure of symmetry occurs for size $v=35$ \cite{Djokovic:JCMCC:1993}. More such examples were found later, see the paper \cite{HKT}.

This problem was circumvented by Goethals and Seidel 
\cite{GS-array:1970} who invented a new array, now known as Goethals-Seidel array or just GS-array, see (\ref{GS-array}) below. This array does not require any of the four circulant blocks to be symmetric. That is its major advantage. If at least 
one of the blocks is of skew type, then one can rearrange 
the blocks to obtain a skew-Hadamard matrix.
Ever since this array has played a very important role in the 
construction of Hadamard matrices and skew-Hadamard matrices.

However, a tool of similar nature for the construction of  symmetric Hadamard matrices was lacking. Such a tool was invented recently by J. Seberry and N. A. Balonin \cite{SB}. They introduced a simple variation of the GS-array to which we 
refer as the {\em Propus array}, see (\ref{Propus-array}) below.
In the paper \cite{SB} it is shown that the symmetry of the array can be easily achieved by demanding that the block {\bf A} be symmetric and that among the remaining three blocks two 
of them are equal, say ${\bf B}={\bf C}$ (an analog of partial symmetry).  This tool has been already used to construct new orders of symmetric Hadamard matrices \cite{SB,DDK:SpecMatC:2015}.

Since the size of a Hadamard matrix or a skew or symmetric 
Hadamard matrix can always be doubled, while preserving the type 
of the matrix, it suffices to construct these matrices for 
orders $4v$ with $v$ odd. 
We show (see Theorem \ref{thm:par-odd}) that for every odd 
integer $v$ there exists at least one propus parameter set. 
Taking this into account, the Propus array can be used, conjecturally, to obtain symmetric Hadamard matrices of order $4v$ for all odd $v$. However there exists a propus parameter set for which there is no cyclic propus family. So far we have 
only one such example namely $(25;10,10,10,10;15)$.

Our purpose is to develop effective numerical algorithms for the search of symmetric Hadamard matrices and subsequent analysis of them and to obtain new orders of such matrices. 
All matrix solutions are classified by using the table of all feasible parameter sets in the range of odd $v<50$. We point out some peculiarities arising from this table. For instance, apart from the Turyn infinite series in which all four circulant blocks are symmetric, there is only one case known so far 
(namely $v=13$) where there exist a propus family with both 
blocks ${\bf A}$ and ${\bf D}$ symmetric, and satisfying 
${\bf B}={\bf C}$ as well.

This paper continues the investigation of the theme of symmetry, considered in the papers \cite{BD:2015-a,BD:2015-b}, and in particular we present for the first time symmetric Hadamard matrices of order 156. In this way, the orders 92, 116, 156, 172 listed as exceptions in \cite[Table 1.52, p. 277]{CK:2007} are all covered by the propus construction.  The next unsolved case is the order 188 which is the object of our further research.

\section{Preliminaries}

Let $G$ be a finite abelian group of order $v$ written additively. A sequence $(X_1,X_2,\ldots,X_m)$ of subsets 
of $G$ is a {\em difference family} if there exists a 
nonnegative integer $\lambda$ such that for any nonzero 
element $a\in G$ there are exactly $\lambda$ triples 
$(x,y,i)\in X_i\times X_i\times\{1,2,\ldots,m\}$ such that 
$x-y=a$. In that case we say that this difference family 
has {\em parameters} $(v;k_1,k_2,\ldots,k_m;\lambda)$, 
where $k_i=|X_i|$ is the cardinality of $X_i$ and that 
the $X_i$ are its {\em base blocks}. A simple counting argument shows that the parameter set of a difference family must satisfy the equality
\begin{equation} \label{eq:uslov-a}
\sum_i k_i(k_i-1)=\lambda(v-1).
\end{equation}
If $G$ is a cyclic group, we say that the difference families of $G$ are {\em cyclic}.

Although the concepts defined below can be defined over arbitrary finite abelian groups, we shall assume in this paper 
that $G$ is a cyclic group of order $v$ and we identify it with the additive group of the ring of integers 
$\bZ_v=\bZ/v\bZ=\{0,1,\ldots,v-1\}$ modulo $v$. We are interested in the difference families consisting of four base blocks having the parameter set $(v;k_1,k_2,k_3,k_4;\lambda)$ such that 
\begin{equation} \label{eq:uslov-b}
\lambda=\sum_i k_i -v.
\end{equation}
For convenience, we shall refer to these parameter sets as 
{\em GS-parameter sets} and to the difference families having these parameters as {\em GS-difference families}. It is a folklore conjecture that for each GS-parameter set there exists 
a cyclic difference family with these parameters.

There is a close relationship between GS-difference families and the quadruples of $\{\pm 1\}$-sequences (also known as {\em binary sequences}) of length $v$ whose periodic autocorellation functions add up to 0 (except at the origin). Let us recall some relevant definitions.

Let $A=(a_0,a_1,\ldots,a_{v-1})$ be an integer sequence of length $v$. We view the indices $0,1,\ldots,v-1$ as elements of $\bZ_v$. 
The {\em periodic autocorrelation function} of $A$ is the function $\paf_A:\bZ_v\to\bZ$ defined by 
\begin{equation} \label{paf-def}
\paf_A(s)=\sum_{i=0}^{v-1} a_i a_{i+s}.
\end{equation}
(The indices should be reduced modulo $v$.)
To $A$ we associate the cyclic matrix $C$ whose first row is $A$ 
itself. We say that $A$ is {\em symmetric} resp. {\em skew} if $a_i=a_{v-i}$ resp. $a_i=-a_{v-i}$ for $i=1,2,\ldots,v-1$. 
Equivalently, $A$ is symmetric if and only if $C$ is a symmetric matrix, and $A$ is skew if and only if $C+C^T=2a_0 I_v$ where 
$T$ denotes the transpose and $I_v$ the identity matrix of order $v$.

To any subset $X\subset\bZ_v$ we associate the binary sequence 
$A=(a_0,a_1,\ldots,a_{v-1})$ where $a_i=-1$ if and only if
$i\in X$. Let $(X_1,X_2,X_3,X_4)$ be a quadruple of subsets of 
$\bZ_v$ with $|X_i|=k_i$ and let $(A_1,A_2,A_3,A_4)$ be their associated binary sequences, respectively. Then it is well known that the $X_i$ form a difference family whose parameter set satisfies the equation (\ref{eq:uslov-b}) if and only if the periodic autocorrelation functions of the $A_i$ add up to 0 (except at the origin).

Let $(A_1,A_2,A_3,A_4)$ be a quadruple of binary sequences of length $v$ whose $\paf$-functions add up to 0, and let 
$(C_1,C_2,C_3,C_4)$ be their associated cyclic matrices. Then by plugging these  matrices into the Goethals-Seidel array:
\begin{equation} \label{GS-array}
H:=\left[ \begin{array}{cccc}
C_1 & C_2R & C_3R & C_4R \\
-C_2R & C_1 & -RC_4 & RC_3 \\
-C_3R & RC_4 & C_1 & -RC_2 \\
-C_4R & -RC_3 & RC_2 & C_1
\end{array} \right],
\end{equation}
we obtain a Hadamard matrix of order $4v$. The matrix
$R$ in (\ref{GS-array}) is the back-circulant identity matrix of order $v$,
$$
\label{Matrix-R}
R=\left[ \begin{array}{ccccc}
0 & 0 & \cdots & 0 & 1 \\
0 & 0 &        & 1 & 0 \\
\vdots &       &   &   \\
0 & 1 &        & 0 & 0 \\
1 & 0 &        & 0 & 0
\end{array} \right].
$$

This is a very powerful method of construction of Hadamard matrices. If $A_1$ is skew then $H$ will be a Hadamard matrix of {\em skew type} ({\em skew-Hadamard matrix}), i.e., a Hadamard matrix such that ${\bf H}+{\bf H}^T=2{\bf I}_{4v}$. 

It was recently observed in \cite{SB} that (after a small twist) one can make use of the GS-array to construct also the symmetric Hadamard matrices. Namely, if we multiply the first column of (\ref{GS-array}) by $-1$ and switch the second and third rows then we obtain the new array, to which we refer as the 
{\em Propus array}, 
\begin{equation} \label{Propus-array}
\left[ \begin{array}{cccc}
-C_1 & C_2R & C_3R & C_4R \\
C_3R & RC_4 & C_1 & -RC_2 \\
C_2R & C_1 & -RC_4 & RC_3 \\
C_4R & -RC_3 & RC_2 & C_1
\end{array} \right]
\end{equation}
This is still a Hadamard matrix. In the special case when 
$A_1$ is symmetric and $A_2=A_3$ this matrix is a symmetric
Hadamard matrix. We say that a GS-parameter set 
$(v;k_1,k_2,k_3,k_4;\lambda)$ is a {\em propus parameter set} 
if $k_2=k_3$, and we say that a difference family 
$(X_1,X_2,X_3,X_4)$ having such parameter set is a {\em propus family} if $X_2=X_3$ and the set $X_1$ or $X_4$ is symmetric.

To summarize, in order to construct a symmetric Hadamard matrix 
of order $4v$ it suffices to construct a propus difference family 
$(X_1,X_2,X_3,X_4)$ in $\bZ_v$. All symmetric Hadamard matrices constructed in this paper use this method. We conjecture that for $v$ odd this method is universal, i.e., for each odd $v>1$ there exists a propus difference family in $\bZ_v$. Theorem \ref{thm:par-odd} below made this conjecture possible.

\section{Existence of parameter sets}

In this section we prove the following theorem.

\begin{theorem} \label{thm:par-odd}
For any odd positive integer $v>1$ there exists a propus parameter set $(v;x,y,y,z;\lambda)$ with $x,y,z<v/2$.
\end{theorem}

Let us first recall an old result of Liouville.
If $x$ is an indeterminate, the polynomial $T_x=x(x+1)/2$ takes
nonnegative integer values at integer points. These values are 
known as {\em triangular numbers}. The {\em ternary triangular form} is a polynomial $aT_x+bT_y+cT_z$, where the coefficients$(a,b,c)$ are positive integers and $x,y,z$ are commuting indeterminates. 
Such form is said to be {\em universal} if it represents all positive integers, i.e., each positive integer is the value of this form at some point $(x_0,y_0,z_0)\in\bZ^3$. Since 
$T_{-x}=T_{x-1}$, we can assume that $x_0,y_0,z_0$ are 
nonnegative. Liouville has proved in 1863 \cite{JL} that there are exactly seven universal ternary triangular forms, assuming 
that $a\le b\le c$. These forms have the coefficients 
$$(1,1,1),~(1,1,2),~(1,1,4),~(1,1,5),~(1,2,2),~(1,2,3),~(1,2,4).$$
This theorem of Liouville generalizes a result of Gauss who proved earlier the universality in the case $a=b=c=1$. We shall use below the fact that the triangular form with coefficients $a=b=1$, $c=2$ is universal.

\begin{proof}
The block sizes $x,y,z$ of our parameter set satisfy the equation 
\begin{equation} \label{eq:par-a}
x(x-1)+2y(y-1)+z(z-1)=\lambda(v-1).
\end{equation}
As $\lambda=x+2y+z-v$ this equation can be written as
\begin{equation} \label{eq:par-b}
(v-2x)^2+2(v-2y)^2+(v-2z)^2=4v.
\end{equation}
Since $v$ is odd, we have $v-2x=2p+1$, $v-2y=2q+1$, $v-2z=2r+1$, 
where $p,q,r$ are integers. Then the above equation becomes \begin{equation} \label{eq:trougao}
T_p+2T_q+T_r=(v-1)/2. 
\end{equation}
By Liouville's result mentioned above, there exist integers $p,q,r$ satisfying this equation. 
Hence, there exist integers $x,y,z$ satisfying the equation 
(\ref{eq:par-b}). If $x<0$ then $v-2x>v$ and the equation  
(\ref{eq:par-b}) implies that $4>v-2x$. This contradicts our 
hypothesis that $v\ge3$. We conclude that $x\ge0$. Similarly, 
we can show that $y,z\ge0$. The equation (\ref{eq:par-a}) implies now that $x+2y+z-v>0$. Hence the theorem is proved.
\qed
\end{proof}

In Appendix A we list the propus parameter sets 
$(v;x,y,y,z;\lambda)$ for odd $v$, $1<v<50$. They are computed by solving the equation (\ref{eq:trougao}) for each of these values $v$. Since we can replace any base block by its complement and permute the blocks, we shall assume that $x,y,z\le(v-1)/2$ 
and $x\ge z$.

For the sake of completeness, let us consider the case when $v$ is even. The result here is quite different, there is an arithmetic condition which rules out the existence of propus parameter sets for some even values of $v$.

\begin{theorem} \label{thm:par-even}
For even positive integer $v$ there exists a propus parameter set $(v;x,y,y,z;\lambda)$ with $x,y,z\le v/2$ if and only if $v$ does not have the form $2^{2k+1}(8m+7)$, where $k$ and $m$ are nonnegative integers.
\end{theorem}
\begin{proof}
The equation (\ref{eq:par-b}) is valid also in the case, i.e.,  when $v$ is even. Then we have $v-2x=2p$, $v-2y=2q$, $v-2z=2r$, where $p,q,r$ are integers. Hence, the equation (\ref{eq:par-b}) can be written as $p^2+2q^2+r^2=v$. By a theorem of Dixon \cite[Theorem V]{LD}, this equation has no integral solution if and only if $v$ has the form $2^{2k+1}(8m+7)$.
One can now easily complete the proof.
\qed
\end{proof}

For instance, this theorem rules out the integers 
$v=14,30,46,56,62,78,94$, i.e., there are no propus parameter sets with these values of the parameter $v$.

\section
{Symmetric Hadamard matrices of order $4\cdot39$}

The smallest order $4v$ for which no symmetric Hadamard
matrix was known prior to this work is $156=4\cdot39$.
There are two parameter sets $(39;17,17,17,15;27)$ and 
$(39;18,16,16,16;27)$ that can be used to construct such matrices. We have constructed many such matrices, but here we record only five pairwise non-equivalent propus families 
for each parameter set.

For the first parameter set, the block ${\bf A}$ is symmetric in the first four families while ${\bf D}$ is symmetric in the last family.

\begin{eqnarray*}
&& (39;17,17,17,15;27) \\
&& [0,2,4,7,8,12,13,18,19,20,21,26,27,31,32,35,37], \\
&& [0,1,2,3,10,14,17,18,19,21,24,26,27,30,32,36,37],\\
&& [0,1,2,3,10,14,17,18,19,21,24,26,27,30,32,36,37],\\
&& [0,1,2,3,4,5,9,11,12,15,26,29,31,33,36]; \\
&& \\
&& [0,2,6,8,9,12,13,18,19,20,21,26,27,30,31,33,37], \\
&& [0,1,2,3,5,6,8,10,14,15,17,18,19,24,28,34,37], \\
&& [0,1,2,3,5,6,8,10,14,15,17,18,19,24,28,34,37], \\
&& [0,1,2,3,5,10,13,16,17,18,22,24,25,28,33]; \\
&& \\
&& [0,3,7,8,9,12,13,17,19,20,22,26,27,30,31,32,36], \\
&& [0,1,2,5,7,8,12,16,20,22,23,25,32,33,34,36,38], \\
&& [0,1,2,5,7,8,12,16,20,22,23,25,32,33,34,36,38], \\
&& [0,1,2,3,6,9,18,20,22,23,30,32,33,34,36]; \\
&& \\
&& [0,3,7,9,13,14,17,18,19,20,21,22,25,26,30,32,36], \\
&& [0,1,2,3,4,7,8,9,10,15,18,20,24,28,29,31,33], \\
&& [0,1,2,3,4,7,8,9,10,15,18,20,24,28,29,31,33], \\
&& [0,1,3,4,8,13,14,16,17,20,23,25,28,35,37];\\
&& \\
&& [0,1,2,3,5,10,12,14,16,17,23,24,28,30,31,36,37],\\
&& [0,1,2,4,9,10,12,13,17,18,22,24,27,30,32,33,37],\\
&& [0,1,2,4,9,10,12,13,17,18,22,24,27,30,32,33,37],\\
&& [0,1,2,3,4,8,14,18,21,25,31,35,36,37,38]. \\
\end{eqnarray*}

For the second parameter set, the block ${\bf A}$ is symmetric in the first family while ${\bf D}$ is symmetric in the other four families. 

\begin{eqnarray*}
&& (39;18,16,16,16;27) \\
&& [3,4,5,7,8,10,12,17,18,21,22,27,29,31,32,34,35,36],\\
&& [0,1,2,3,8,9,17,19,21,23,26,29,32,35,36,37],\\
&& [0,1,2,3,8,9,17,19,21,23,26,29,32,35,36,37],\\
&& [0,1,2,4,5,6,10,11,13,14,21,22,27,29,33,36];\\
&& \\
&& [0,1,2,6,7,9,10,12,15,17,20,21,24,28,29,31,33,37],\\
&& [0,1,2,3,7,13,18,20,21,24,27,28,32,34,36,37],\\
&& [0,1,2,3,7,13,18,20,21,24,27,28,32,34,36,37],\\
&& [2,7,8,9,15,17,18,19,20,21,22,24,30,31,32,37]; \\
&& \\
&& [0,1,2,6,8,9,10,12,14,16,18,19,27,30,32,33,36,37],\\
&& [0,1,2,3,7,8,11,19,21,24,26,27,28,31,33,36],\\
&& [0,1,2,3,7,8,11,19,21,24,26,27,28,31,33,36],\\
&& [2,4,5,6,7,11,17,18,21,22,28,32,33,34,35,37]; \\
&& \\
&& [0,1,2,8,9,15,16,18,19,21,23,26,28,30,31,32,34,35],\\
&& [0,1,2,5,6,14,17,20,22,24,25,27,28,29,31,35],\\
&& [0,1,2,5,6,14,17,20,22,24,25,27,28,29,31,35],\\
&& [1,7,10,11,12,13,17,19,20,22,26,27,28,29,32,38]; \\
&& \\
&& [0,1,2,3,4,6,8,9,10,12,18,23,24,28,30,31,32,35],\\
&& [0,1,2,5,6,8,11,13,15,17,24,27,29,30,36,37],\\
&& [0,1,2,5,6,8,11,13,15,17,24,27,29,30,36,37],\\
&& [4,5,8,9,10,13,15,16,23,24,26,29,30,31,34,35]. \\
\end{eqnarray*}

\section{Description of the algorithm}
\label{Algoritam}

Let us first describe the algorithm for the search of periodic 
Golay pairs, a somewhat simpler problem. The search we have in mind is a non-exhaustive search which uses a random number generator to create the sequences. 

The periodic Golay pairs of length $v$ are pairs of $\{\pm1\}$-sequences 
${\bf a}=(a_0,a_1,\ldots,a_{v-1})$ and 
${\bf b}=(b_0,b_1,\ldots,b_{v-1})$ whose $\paf$ functions have sum $0$ except at the origin. (We shall ignore the value of the $\paf$ functions at the origin.) These pairs exist only for even values of $v$ (excluding the trivial case $v=1$). The number of indices $i$ such that $a_i=-1$ is fixed, and we denote it by 
$k_1$. Similarly, $k_2$ is the number of $-1$ terms in ${\bf b}$. 
Since the $\paf$ values of a sequence are symmetric, i.e., 
$\paf_{\bf a}(s)=\paf_{\bf a}(v-s)$ for $s=1,2,\ldots,v-1$, it suffices to compute and record these values for $1\le s\le v/2$.

The very simple and time consuming algorithm can be described as follows. First it generates just one random a-sequence having exactly $k_1$ terms $-1$ and computes its $\paf$ function. 
Next, it computes a bunch of (say w) random b-sequences having exactly $k_2$ terms $-1$. At the same time it computes their $\paf$ values and checks whether the sum of the $\paf$ functions of the a-sequence and the b-sequence is $0$. (The required memory for this is negligible.) This complets  one basic step. This step is then repeated as long as desired.

A more effective algorithm generates a collection of, say, $w$ binary sequences ${\bf a}$ having exactly $k_1$ terms $-1$ 
and records them together with the $\paf$ values in a table. Another table also of size $w$ is used to generate and record a collection of binary sequences ${\bf b}$ having exactly $k_2$ terms $-1$ and their negated $\paf$ values.
The two tables of size $w$ make it possible to make quickly 
$w^2$ comparisons. 

The second method performs faster because it computes only $2w$ sequences (and their $\paf$ values) in order to check $w^2$ pairs for matching, while $w$ steps of the simple method has to compute $w(w+1)$ sequences to check $w^2$ pairs for matching.
So, the saving is in the number of sequences that one has to
generate and compute the paf-values: $w(w+1)$ for the brute force method and $2w$ for the second method.

However, making two big tables is not feasible as the active 
memory is limited. To handle this problem, one of the authors 
proposed and implemented the following solution. The tables 
of data are replaced by trees having a fixed number of branches. 
Each branch can hold at most $w$ records of data to which we 
refer as {\em leaves} of that branch. A random number generator is used to generate data and a hash-function, $f$, is used 
to distribute the data and store them into the branches. 
After generating a sequence say ${\bf a}$ and computing its 
$\paf$ values, the hash function is evaluated at the 
$\paf$ values which gives the numerical label $f(\paf_{\bf a})$ 
of the branch where the data will be stored. In the case of the ${\bf b}$ sequence, the $\paf$ values are negated just before storing them into the chosen branch.

If $f$ takes different values at the functions $\paf_{\bf a}$ and $-\paf_{\bf b}$ then $\paf_{\bf a}\ne -\paf_{\bf b}$, but the converse fails. Consequently, no comparisons need to be made between the ${\bf a}$-leaves and ${\bf b}$-leaves belonging to branches having different labels. For that reason, this third 
method is much more effective than the second one.

A big crown of size $M=2^m$ gives the function
$$
f=\sum_{i=1}^m {\rm sign}(\paf(i))\cdot 2^{i-1}
$$
whose coefficients are the signatures 
${\rm sign}(\paf(i))\in\{0,1\}$ of the first $m$ $\paf$ 
values. (We take that $1$ corresponds to positive values 
of the $\paf$ function.)  This definition can be modified 
by using the ternary function 
${\rm sign}(\paf(i))\in\{0,1,-1\}$ which distinguishes 
$0$ and the signs of the nonzero $\paf$ values, and 
adding $2^m$ if necessary to make the label positive.

Abstract ``ideal'' hash-function gives strictly uniform distribution of leaves over the branches.

When using binary representation of integers in computers memory it is beneficial to use logic operations of the iteration 
formula
$$
F=((F~{\rm shl}~1)~{\rm or}~(F~{\rm shr}~31)~{\rm and}~1),
~i=1,\ldots,v/2.
$$
The symbols ``shl'' and ``shr'' denote left and right shift of 
the binary code for the indicated amount, the computations 
begin with the value $F=0$ and terminate with the restriction 
$f=F~{\rm mod}~ M$, which gurantees that the size of the crown will be $M=2^m$.

\section{Parameter sets for symmetric Hadamard matrices}
\label{Parametri}

We list here the propus parameter sets $(v;x,y,y,z;\lambda)$ with $v$ odd in the range $v<50$ such that $x,y,z<v/2$ and 
$x\ge z$. The cyclic propus families consisting of four base blocks $A,B,C,D\subseteq\bZ_v$ having sizes $x,y,y,z$, respectively, and such that $B=C$ and either $A$ or $D$ is symmetric give symmetric Hadamard matrices of order $4v$. If $x=z\ne y$ then we also include in our list the parameter set $(v;y,x,x,y;\lambda)$ indicating that the two blocks of size $x$ are required to be equal. In that case we treat these two parameter sets as different propus parameter sets.

The four base blocks are denoted by $A,B,C,D$. In all propus families mentioned below we require that $B=C$. If we know that there is such a family with symmetric block $A$, we indicate this by the symbol $A$, and similarly for the block $D$. If we know that there exists a family with both $A$ and $D$ symmetric, then we write the symbol $AD$. If there are no families with $A$ or $D$ symmetric, we write ``No'' after the parameter set. Finally, the question mark means that the existence of families with $A$ or $D$ symmetric remains undecided.

The symbol $T$ indicates that the parameter set belongs to the Turyn series of Williamson matrices. We note that $T$ implies 
$AD$. Further, the symbol $X$ indicates that the parameter set belongs to another infinite series (see 
\cite[Theorem 5]{DDK:SpecMatC:2015}) which is based on the 
paper \cite{XXSW} of Xia, Xia, Seberry, and Wu. In our list 
below $X$ implies $D$. More precisely, for a difference family
$A,B,C,D$ in the $X$-series two blocks are equal, say $B=C$, 
and one of the remaining blocks is skew, block $A$ in our list, 
and the last one is symmetric, block $D$. We remark that a difference family in the $X$-series gives a skew and a symmetric Hadamard matrix of order $4v$.

There are only two propus parameter sets, $(5;1,2,2,1;1)$ and $(25;10,10,10,10;15)$, for which there are no cyclic propus difference families. While for the fomer set this claim can be easily proved, for the latter set it was checked by performing an exhaustive search. There is a possibility that a propus family with parameters $(25;10,10,10,10;15)$ may exist in 
$\bZ_5\times\bZ_5$.

\newpage
\begin{center}
Table of propus paramater sets with odd $v<50$
\end{center}

$$
\begin{array}{llll}
(3;1,1,1,0;0) & AD,T,X  & (5;1,2,2,1;1) & {\rm No} \\ 
(5;2,1,1,2;1) & AD,T,X  & (7;3,2,2,2;2) & AD,T  \\ 
(7;3,3,3,1;3) & D,X  & (9;3,3,3,3;3) & A,D  \\ 
(9;3,4,4,2;4) & AD,T  & (11;5,4,4,3;5) & A,D,X  \\ 
(13;4,6,6,4;7) & A,D  & (13;5,5,5,4;6) & AD,T  \\ 
(13;6,4,4,6;7) & AD  & (13;6,6,6,3;8) & A,D  \\ 
(15;6,7,7,4;9) & A,D  & (15;7,5,5,6;8) & AD,T  \\ 
(17;6,7,7,6;9) & A,D  & (17;7,6,6,7;9) & A,D  \\ 
(17;8,7,7,5;10) & A,D,X  & (19;7,9,9,6;12) & AD,T  \\ 
(19;8,8,8,6;11) & A,D  & (19;9,7,7,7;11) & A,D  \\ 
(21;9,8,8,8;12) & AD,T  & (21;10,10,10,6;15) & A,D,X  \\ 
(23;9,10,10,8;14) & A,D  & (23;10,11,11,7;16) & A,D  \\ 
(25;9,12,12,9;17) & A,D  & (25;10,10,10,10;15) & {\rm No}  \\ 
(25;12,9,9,12;17) & AD,T  & (25;12,10,10,9;16) & A,D  \\ 
(25;12,11,11,8;17) & A,D  & (27;11,13,13,9;19) & A,D  \\ 
(27;12,11,11,10;17) & A,D  & (27;12,12,12,9;18) & A,D  \\ 
(27;13,10,10,12;18) & AD,T,X  & (29;11,13,13,11;19) & A,D  \\ 
(29;13,11,11,13;19) & A,D  & (31;13,13,13,12;20) & AD,T  \\ 
(31;13,14,14,11;21) & A,D  & (31;15,12,12,13;21) & A,D  \\ 
(31;15,15,15,10;24) & A,D  & (33;13,16,16,12;24) & A,D  \\ 
(33;15,13,13,14;22) & A,D  & (33;15,16,16,11;25) & A,D  \\ 
(33;16,14,14,12;23) & A,D,X  & (35;16,15,15,13;24) & A,D  \\ 
(35;17,16,16,12;26) & A,D,X  & (37;15,16,16,15;25) & A,D  \\ 
(37;15,17,17,14;26) & AD,T  & (37;16,15,15,16;25) & A,D  \\ 
(37;16,18,18,13;28) & A,D  & (37;17,17,17,13;27) & A,D  \\ 
(37;18,15,15,15;26) & A,D  & (39;17,17,17,15;27) & A,D  \\ 
(39;18,16,16,16;27) & A,D  & (41;16,20,20,16;31) & A,D  \\ 
(41;18,19,19,15;30) & A,D  & (41;20,16,16,20;31) & AD,T,X  \\ 
(43;18,21,21,16;33) & ?  & (43;19,18,18,18;30) & ?  \\ 
(43;21,17,17,20;32) & ?  & (43;21,19,19,16;32) & ?  \\ 
(43;21,21,21,15;35) & D  & (45;18,21,21,18;33) & ?  \\ 
(45;19,20,20,18;32) & AD,T  & (45;21,18,18,21;33) & ? \\ 
(45;21,20,20,17;33) & ?  & (45;21,22,22,16;36) & ?  \\ 
(45;22,19,19,18;33) & D,X  & (47;20,22,22,18;35) & ?  \\ 
(47;22,20,20,19;34) & ?  & (47;23,19,19,21;35) & ?  \\ 
(47;23,22,22,17;37) & ?  & (49;21,21,21,21;35) & ?  \\ 
(49;22,22,22,19;36) & ?  & (49;22,24,24,18;39) & ?  \\ 
(49;23,20,20,22;36) & AD,T  & (49;23,23,23,18;38) & ?  \\ 
\end{array}
$$

\section{Acknowledgements}
The third author wishes to acknowledge generous support by NSERC.
His work was made possible by the facilities of the Shared Hierarchical Academic Research Computing Network (SHARCNET) and Compute/Calcul Canada.

\section{Appendix A}

In order to justify our claims made in section \ref{Parametri} regarding the propus parameter sets, we give the examples of the propus families having the required properties. In all cases the blocks $B$ and $C$ are equal, and to save space we omit the block $C$. The families are terminated by semicolons.

\begin{eqnarray*}
&& (9;3,3,3,3;3) \\
&& [0,1,8],~ [0,2,5],~ [0,1,4]; \\
&& \\
&& (11;5,4,4,3;5) \\
&& [0,2,5,6,9],~ [0,1,2,8],~ [0,2,8]; \\
&& \\
&& (13;4,6,6,4;7) \\
&& [3,5,8,10],~ [0,1,2,3,6,10],~ [0,1,5,7]; \\
&& \\
&& (13;6,4,4,6;7) \\
&& [2,5,6,7,8,11],~ [0,1,4,6],~ [1,3,4,9,10,12]; \\
&& \\
&& (13;6,6,6,3;8) \\
&& [1,4,5,8,9,12],~ [0,1,2,4,6,7],~ [0,2,5]; \\
&& [0,1,3,4,6,9], [0,1,2,8,9,11],~ [0,4,9]; \\
&& \\
&& (15;6,7,7,4;9) \\
&& [1,6,7,8,9,14],~ [0,1,2,4,5,7,11],~ [0,3,6,10]; \\
&& [0,2,4,5,10,12],~ [0,1,2,4,9,10,13],~ [5,6,9,10]; \\
&& \\
&& (17;6,7,7,6;9) \\
&& [2,5,6,11,12,15],~ [0,1,2,3,5,8,13],~ [0,1,7,9,11,15]; \\
&& [0,6,7,8,9,10,11],~ [0,1,5,7,10,13],~ [0,1,2,6,8,11,15]; \\
&& \\
&& (17;8,7,7,5;10) \\
&& [2,3,5,6,11,12,14,15],~ [0,1,3,4,11,13,15],~ [0,1,5,6,12]; \\
\end{eqnarray*}

\begin{eqnarray*}
&& (19;8,8,8,6;11) \\
&& [1,2,3,9,10,16,17,18], \\
&& [0,1,3,9,12,13,15,17], \\
&& [0,1,6,7,10,15]; \\
&& [0,1,2,4,6,9,12,13], \\
&& [0,1,2,5,6,12,15,17], \\
&& [4,6,7,12,13,15]; \\
&& \\
&& (19;9,7,7,7;11) \\
&& [0,1,2,3,7,12,16,17,18], \\
&& [0,1,3,7,11,12,14], \\
&& [0,1,3,6,9,13,15]; \\
&& [0,1,2,7,12,15,16,17,18], \\
&& [0,1,4,6,11,13,14], \\
&& [0,2,6,9,10,13,17]; \\
&& \\
&& (21;10,10,10,6;15) \\
&& [1,2,3,5,10,11,16,18,19,20], \\
&& [0,1,3,4,6,8,11,12,13,18], \\
&& [0,1,2,6,12,19]; \\
&& \\
&& (23;9,10,10,8;14) \\
&& [0,2,3,6,10,13,17,20,21], \\
&& [0,2,4,5,6,7,12,13,18,21], \\
&& [2,3,6,11,12,14,15,16]; \\
&& [0,1,4,9,14,17,19,21,22], \\
&& [0,5,9,11,12,13,14,16,20,22], \\
&& [2,5,6,11,12,17,18,21]; \\
&& \\
&& (23;10,11,11,7;16) \\
&& [1,3,4,9,10,13,14,19,20,22], \\
&& [1,3,4,6,7,8,9,15,18,19,22], \\
&& [1,3,4,5,10,18,20]; \\
\end{eqnarray*}

\begin{eqnarray*}
&& [1,2,5,11,12,15,16,18,19,20], \\
&& [1,3,4,5,6,7,13,16,18,20,21], \\
&& [0,5,7,11,12,16,18]; \\
&& \\
&& (25;9,12,12,9;17) \\
&& [0,1,5,8,10,15,17,20,24], \\
&& [0,1,3,9,12,13,14,16,17,19,20,24], \\
&& [1,7,13,14,15,17,18,20,24]; \\
&& \\
&& (25;12,10,10,9;16) \\
&& [1,2,3,4,10,12,13,15,21,22,23,24], \\
&& [0,5,10,14,15,17,18,21,23,24], \\
&& [2,4,8,12,14,16,19,20,24]; \\
&& [0,1,7,12,14,15,17,18,20,21,22,23], \\
&& [2,3,4,6,11,12,13,16,18,24], \\
&& [0,2,6,9,10,15,16,19,23]; \\
&& \\
&& (25;12,11,11,8;17) \\
&& [3,4,5,6,9,11,14,16,19,20,21,22], \\
&& [0,1,9,10,13,14,17,19,20,21,23], \\
&& [2,9,13,15,17,20,22,23]; \\
&& [0,4,5,8,11,12,13,14,15,16,17,22], \\
&& [3,4,7,8,10,13,15,19,21,22,23], \\
&& [2,3,5,7,18,20,22,23]; \\
&& \\
&& (27;11,13,13,9;19) \\
&& [0,1,2,4,8,12,15,19,23,25,26], \\
&& [3,4,5,6,8,9,11,15,16,18,20,24,25], \\
&& [0,8,9,10,13,16,18,19,22]; \\
&& [2,3,4,8,9,12,13,15,17,19,22], \\
&& [2,3,6,9,11,17,18,19,20,21,23,25,26], \\
&& [0,1,3,8,13,14,19,24,26]; \\
&& \\
\end{eqnarray*}

\begin{eqnarray*}
&& (27;12,11,11,10;17) \\
&& [1,4,5,6,8,9,18,19,21,22,23,26], \\
&& [5,6,9,11,14,16,17,18,20,25,26], \\
&& [1,2,4,5,6,10,12,18,22,26]; \\
&& [0,1,2,5,6,8,11,15,16,17,20,22], \\
&& [1,2,7,9,15,17,18,19,21,22,26], \\
&& [1,2,3,6,12,15,21,24,25,26]; \\
&& \\
&& (27;12,12,12,9;18) \\
&& [1,7,9,10,11,13,14,16,17,18,20,26], \\
&& [0,2,3,4,7,8,9,14,18,19,21,22], \\
&& [0,2,4,6,9,12,15,20,26]; \\
&& [2,3,6,7,8,9,12,14,15,20,24,26], \\
&& [1,3,4,8,11,12,19,20,22,24,25,26], \\
&& [0,2,5,12,13,14,15,22,25]; \\
&& \\
&& (29;11,13,13,11;19) \\
&& [0,2,5,7,13,14,15,16,22,24,27], \\
&& [3,6,7,8,9,10,13,14,17,18,20,22,26], \\
&& [0,1,3,4,9,12,15,17,22,23,28]; \\
&& \\
&& (29;13,11,11,13;19) \\
&& [0,2,8,9,12,13,14,15,16,17,20,21,27], \\
&& [1,2,3,10,12,16,19,20,22,25,27], \\
&& [0,1,2,4,5,7,8,9,13,17,18,20,23]; \\
&& \\
&& (31;13,14,14,11;21) \\
&& [0,1,2,5,8,11,12,19,20,23,26,29,30], \\
&& [0,3,8,14,15,16,17,19,21,22,25,27,29,30], \\
&& [0,1,4,6,11,15,16,20,22,24,30]; \\
&& [2,3,5,7,9,10,12,15,16,17,28,29,30], \\
&& [0,1,3,4,8,12,15,16,18,24,25,26,27,29], \\
&& [0,1,6,10,12,15,16,19,21,25,30]; \\
\end{eqnarray*}

\begin{eqnarray*}
&& (31;15,12,12,13;21) \\
&& [0,2,4,7,10,11,12,15,16,19,20,21,24,27,29], \\
&& [2,3,5,12,14,15,19,20,25,26,27,30], \\
&& [6,8,9,11,12,13,14,15,16,18,24,26,28]; \\
&& [0,1,6,9,13,14,15,16,17,18,22,25,30], \\
&& [2,4,6,9,10,13,15,20,23,25,26,29], \\
&& [0,1,2,4,5,6,9,10,11,12,22,23,24,28,30]; \\
&& \\
&& (31;15,15,15,10;24) \\
&& [0,1,2,7,10,11,14,15,16,17,20,21,24,29,30],\\
&& [0,3,5,7,10,12,13,14,15,16,18,19,23,24,30],\\
&& [0,2,4,6,12,14,22,25,26,28];\\
&& [0,4,6,8,10,11,15,16,18,20,21,23,24,26,28],\\
&& [2,4,5,6,8,11,12,13,16,20,24,25,26,27,30],\\
&& [2,11,12,13,14,17,18,19,20,29];\\
&& \\
&& (33;13,16,16,12;24) \\
&& [0,1,2,4,7,12,14,19,21,26,29,31,32], \\
&& [1,2,3,6,9,10,12,13,14,17,18,19,23,24,26,32], \\
&& [1,4,5,11,13,14,15,17,19,20,28,32]; \\
&& [0,1,4,8,10,14,16,18,19,25,27,30,32], \\
&& [0,5,7,9,10,13,14,15,16,17,18,21,22,27,28,30], \\
&& [2,3,4,6,13,16,17,20,27,29,30,31]; \\
&& \\
&& (33;15,13,13,14;22) \\
&& [0,2,3,5,9,10,13,14,19,20,23,24,28,30,31], \\
&& [0,1,2,3,4,8,14,17,20,24,26,28,29], \\
&& [0,2,3,4,5,6,7,14,19,20,22,25,27,29]; \\
&& [0,1,2,4,7,11,13,16,18,26,27,28,29,30,32], \\
&& [0,1,2,7,9,10,11,15,16,19,22,29,32], \\
&& [1,4,5,9,11,13,16,17,20,22,24,28,29,32]; \\
&& \\
\end{eqnarray*}

\begin{eqnarray*}
&& (33;15,16,16,11;25) \\
&& [0,1,4,7,8,9,10,13,20,23,24,25,26,29,32], \\
&& [0,1,4,6,8,10,11,12,15,17,18,19,22,23,24,31], \\
&& [1,6,7,8,10,11,14,16,26,29,31]; \\
&& [0,1,5,8,12,14,15,17,18,20,22,25,26,29,32], \\
&& [0,1,3,8,13,14,15,19,21,22,23,25,27,30,31,32], \\
&& [0,3,10,13,14,15,18,19,20,23,30]; \\
&& \\
&& (33;16,14,14,12;23) \\
&& [1,2,4,6,10,14,15,16,17,18,19,23,27,29,31,32], \\
&& [0,2,9,10,11,13,15,16,19,20,21,22,25,28], \\
&& [1,4,8,9,11,15,16,18,25,26,30,31]; \\
&& \\
&& (35;16,15,15,13;24) \\
&& [1,3,4,7,9,10,11,15,20,24,25,26,28,31,32,34], \\
&& [0,2,6,7,9,11,14,17,18,19,28,29,32,33,34], \\
&& [0,1,3,4,10,15,17,23,26,28,29,30,32]; \\
&& [0,6,9,11,14,16,18,21,22,25,26,27,29,32,33,34], \\
&& [0,2,4,8,11,12,13,18,23,26,29,30,31,32,33], \\
&& [0,1,9,10,13,14,16,19,21,22,25,26,34]; \\
&& \\
&& (35;17,16,16,12;26) \\
&& [0,1,2,6,9,12,14,16,17,18,19,21,23,26,29,33,34], \\
&& [1,5,12,13,14,15,17,20,21,23,25,26,27,30,31,34], \\
&& [4,10,11,15,16,22,23,24,25,27,31,34]; \\
&& \\
&& (37;15,16,16,15;25) \\
&& [0,2,3,5,10,11,12,16,21,25,26,27,32,34,35], \\
&& [0,1,2,5,9,10,12,13,15,16,22,28,30,33,34,35], \\
&& [0,2,4,6,10,11,12,13,16,19,20,22,30,31,33]; \\
&& \\
&& (37;16,15,15,16;25) \\
&& [1,4,5,6,9,12,13,14,23,24,25,28,31,32,33,36], \\
&& [0,2,3,4,6,7,13,14,16,19,24,28,30,35,36], \\
&& [1,5,7,8,11,13,14,15,19,26,28,30,32,33,35,36]; \\
&& \\
\end{eqnarray*}

\begin{eqnarray*}
&& (37;16,18,18,13;28) \\
&& [1,3,6,11,12,16,17,18,19,20,21,25,26,31,34,36], \\
&& [1,2,7,11,14,15,17,18,23,25,26,27,28,29,31,32,34,36], \\
&& [1,2,3,4,11,15,17,21,22,26,29,32,33]; \\
&& [1,2,4,6,9,11,12,14,18,19,20,21,22,23,24,36], \\
&& [2,3,5,6,9,10,11,12,16,17,19,20,21,22,27,31,33,35], \\
&& [0,2,6,7,11,14,17,20,23,26,30,31,35]; \\
&& \\
&& (37;17,17,17,13;27) \\
&& [0,2,3,5,6,8,9,15,16,21,22,28,29,31,32,34,35], \\
&& [2,4,5,7,8,9,15,16,18,19,20,23,24,25,27,29,33], \\
&& [2,3,10,12,19,20,22,24,27,29,31,34,35]; \\
&& [1,2,3,4,9,11,13,15,16,19,20,21,22,24,25,27,28], \\
&& [2,6,9,10,13,15,19,24,25,26,28,29,31,33,34,35,36], \\
&& [0,3,5,11,12,13,17,20,24,25,26,32,34]; \\
&& \\
&& (37;18,15,15,15;26) \\
&& [3,4,5,6,7,9,11,13,16,21,24,26,28,30,31,32,33,34], \\
&& [0,4,7,13,16,17,18,19,22,23,24,29,30,32,33], \\
&& [1,5,9,10,12,15,18,22,23,24,26,28,30,31,33]; \\
&& [0,2,4,6,9,11,12,13,15,16,17,18,22,24,25,30,35,36], \\
&& [1,4,9,13,16,17,20,21,22,23,24,31,32,33,36], \\
&& [0,6,8,10,11,14,16,17,20,21,23,26,27,29,31]; \\
&& \\
&& (39;17,17,17,15;27) \\
&& [0,5,6,8,10,14,15,17,18,21,22,24,25,29,31,33,34],\\
&& [0,3,5,6,12,13,14,16,17,18,22,27,30,33,35,37,38],\\
&& [1,2,3,5,6,8,12,13,14,15,26,31,32,34,38];\\
&& [2,3,7,9,10,15,16,18,19,20,21,23,28,30,32,34,35],\\
&& [0,3,5,6,10,12,13,14,16,21,22,24,25,29,30,34,36],\\
&& [0,1,2,3,4,8,14,18,21,25,31,35,36,37,38];\\
&& \\
\end{eqnarray*}

\begin{eqnarray*}
&& (39;18,16,16,16;27) \\
&& [3,4,5,7,8,10,12,17,18,21,22,27,29,31,32,34,35,36],\\
&& [0,3,4,5,7,8,9,10,15,16,24,26,28,30,33,36],\\
&& [2,3,4,6,7,8,12,13,15,16,23,24,29,31,35,38];\\
&& [1,2,3,9,10,16,17,19,20,22,24,27,29,31,32,33,35,36],\\
&& [1,5,6,7,10,11,19,22,25,27,29,30,32,33,34,36],\\
&& [1,7,10,11,12,13,17,19,20,22,26,27,28,29,32,38];\\
&& \\
&& (41;16,20,20,16;31) \\
&& [1,2,3,9,11,15,19,20,21,22,26,30,32,38,39,40],\\
&& [0,3,9,11,14,15,16,19,22,23,24,25,26,28,29,30,32,35,37,40],\\
&& [0,4,5,7,14,16,18,19,21,23,24,31,32,37,38,40];\\
&& \\
&& (41;18,19,19,15;30) \\
&& [4,5,7,8,10,11,15,16,17,24,25,26,30,31,33,34,36,37], \\
&& [3,5,8,9,11,12,16,17,18,19,21,22,23,26,28,30,33,34,36], \\
&& [0,2,3,5,6,11,15,20,22,24,26,27,34,35,39]; \\
&& [3,5,6,9,10,11,13,15,16,26,27,31,33,35,36,38,39,40],\\
&& [1,2,5,8,9,11,13,14,15,16,18,20,23,24,29,30,31,32,40],\\
&& [0,2,4,5,10,15,18,19,22,23,26,31,36,37,39];\\
&& \\
&& (43;21,21,21,15;35) \\
&& [0,1,2,3,4,8,9,12,14,19,22,23,26,28,29,31,32,34,38,39,41],\\
&& [1,4,6,9,10,11,13,14,15,16,17,21,23,24,25,31,35,36,38,40,41],\\
&& [0,7,9,13,14,15,17,18,25,26,28,29,30,34,36];\\
\end{eqnarray*}

The last example consists of a D-optimal design (blocks $A$ and $D$) and two copies of the Paley difference set in $\bZ_{43}$ 
(blocks $B=C$). It is taken from the paper \cite{DDK:SpecMatC:2015}.

\end{document}